\documentclass[10pt]{amsart}
\usepackage{amsthm, amsfonts, amssymb, amsmath, graphicx, enumitem, bbold, caption}
\usepackage{lmodern}
\usepackage[margin=4cm]{geometry}

\newcommand{\R}{\mathbb{R}}

\newcommand{\G}{\Gamma}
\newcommand{\F}{\mathbb{F}}

\newcommand{\e}{\epsilon}

\newcommand{\la}{\langle}
\newcommand{\ra}{\rangle}
\newcommand{\bd}{\partial}

\newcommand{\xalg}{\rtimes_{\mathrm{alg}}}
\newcommand{\xr}{\rtimes_{\mathrm{r}}}

\newtheorem{thm}{Theorem}[section]

\theoremstyle{definition}

\newtheorem*{ack}{Acknowledgments}

\begin{document}      
\title{Two applications of strong hyperbolicity}
\author{Bogdan Nica}
\address{\newline Department of Mathematics and Statistics \newline McGill University, Montreal}
\date{February 25, 2017}

\begin{abstract}
We present two analytic applications of the fact that a hyperbolic group can be endowed with a strongly hyperbolic metric. The first application concerns the crossed-product C$^*$-algebra defined by the action of a hyperbolic group on its boundary. We construct a natural time flow, involving the Busemann cocycle on the boundary. This flow has a natural KMS state, coming from the Hausdorff measure on the boundary, which is furthermore unique when the group is torsion-free. The second application is a short new proof of the fact that a hyperbolic group admits a proper isometric action on an $\ell^p$-space, for large enough $p$. 
\end{abstract}

\maketitle

\section{Introduction}
Hyperbolicity, in the sense of Gromov, is a coarse notion of negative curvature for metric spaces. In turn, a hyperbolic group is a group which admits a proper and cocompact isometric action on a geodesic hyperbolic space. Such a space is said to be a geometric model for the group. Hyperbolic groups form a large class of groups, and they have received a lot of attention--usually from an algebraic and geometric perspective. Herein, the aims are mostly analytic. 

A sharp notion of negative curvature for metric spaces is captured by the CAT$(-1)$ condition. This condition implies, and predates, hyperbolicity. Gromov's Jugendtraum \cite[p.193]{Gro}, that every hyperbolic group admits a geometric model which is CAT$(-1)$, is still wildly open. It is expected to fail, but no counterexamples are known. Let us mention, however, that the past decade has seen great strides in the CAT$(0)$ direction. We now understand that an extraordinary number of hyperbolic groups act on CAT$(0)$ cube complexes.

The search for enhanced geometric models of hyperbolic groups is often motivated by analytic needs. We use the term `enhanced hyperbolicity' as a broad and informal way of describing hyperbolicity with additional CAT$(-1)$ properties. Such desirable properties depend on the specific context. In \cite{NS}, we introduced the metric notion of \emph{strong hyperbolicity}. We find this idea satisfactory on two accounts. Firstly, it is an intermediate metric notion between the CAT$(-1)$ condition and hyperbolicity, which grants the additional CAT$(-1)$ properties that, so far, have come up in analytic applications. Secondly, it turns out that every hyperbolic group admits a geometric model which is strongly hyperbolic. We briefly discuss strong hyperbolicity in Section~\ref{sec: strong} below, and we refer to  \cite{NS} for more details.

The purpose of this note is to further illustrate the use of strong hyperbolicity in studying analytic aspects of hyperbolic groups. The first application concerns the C$^*$-crossed product $C(\bd\G)\rtimes \G$ defined by the action of a hyperbolic group $\G$ on its boundary $\bd \G$. We use strong hyperbolicity to construct a natural $\R$-flow on $C(\bd\G)\rtimes \G$, from the Busemann cocycle on the boundary. We show that the Hausdorff measure on the boundary defines a KMS state for this Busemann flow, with inverse temperature equal to the Hausdorff dimension of the boundary. Furthermore, this is the unique KMS state for the flow when $\G$ is torsion-free. Previously, these facts were known in two particular cases: for free groups \cite{CM}, respectively for uniform lattices in $\mathrm{SO}(n,1)$ \cite{Lott}. Compare also \cite{IK}.

The second application is a short new proof of the fact that a hyperbolic group admits a proper isometric action on an $\ell^p$-space, for large enough $p\in [1,\infty)$. This result is due to Yu \cite{Yu}, and different proofs have been subsequently offered in \cite{Bou, Nic, AL}. The argument explained in Section~\ref{sec: haag} provides a link between Haagerup's original construction for free groups \cite{Haa}, and the boundary construction of \cite{Nic}.

\section{Strong hyperbolicity}\label{sec: strong}
\subsection{Strongly hyperbolic spaces} Let $X$ be a metric space. We write $|x,y|$ for the distance between two points $x,y\in X$. Recall that the Gromov product with respect to a basepoint $o$ is defined by the formula 
\begin{align*}
\la x,y\ra_o=\tfrac{1}{2}\big(|o,x|+|o,y|-|x,y|\big).
\end{align*} 

The metric space $X$ is said to be \emph{strongly hyperbolic} if the Gromov product satisfies
\begin{align*}
e^{-\la x,y\ra_o}\leq e^{-\la x,z\ra_o}+e^{-\la z,y\ra_o}
\end{align*}
for all $x,y,z,o\in X$. (Compare with the original definition \cite[Def.4.1]{NS}, involving an additional `visual' parameter $\e>0$. The present definition is the normalized case when $\e=1$, and this can always be achieved by rescaling the metric.) 

It is easily checked that a strongly hyperbolic space is, in particular, hyperbolic in the usual, Gromov sense. On the other hand, a CAT$(-1)$ space is strongly hyperbolic \cite[Thm.5.1]{NS}. To put it differently, strong hyperbolicity is a weak CAT$(-1)$ condition. For the purposes of this paper, a useful consequence of strong hyperbolicity is the following \cite[Thm.4.2]{NS}:

\begin{thm}\label{thm: good}
Let $X$ be a strongly hyperbolic space, and let $o\in X$ be a basepoint. Then the Gromov product $\la \cdot,\cdot\ra_o$ extends continuously to the bordification $X\cup\bd X$, and $e^{-\la \cdot,\cdot\ra_o}$ is a compatible metric on the boundary $\bd X$.
\end{thm}

\subsection{Strongly hyperbolic metrics for hyperbolic groups} Let $\G$ be a hyperbolic group. To avoid trivialities, we will always assume that $\G$ is non-elementary. A metric on $\G$ is said to be \emph{admissible} if it enjoys the following properties:

\begin{itemize}
\item[(i)] it is equivariant: $|gx,gy|=|x,y|$ for all $g,x,y\in \G$;
\item[(ii)] it is roughly geodesic: there is a constant $C\geq 0$, so that for every pair of points $x,y\in \G$ there is a (not necessarily continuous) map $\gamma: [a,b]\to \G$ satisfying $\gamma(a)=x$, $\gamma(b)=y$, and $|s-t|-C\leq |\gamma(s),\gamma(t)|\leq |s-t|+C$ for all $s,t\in [a,b]$; 
\item[(iii)] it is quasi-isometric to any word metric on $\G$.
\end{itemize}
An admissible metric on $\G$ is hyperbolic, since hyperbolicity is a quasi-isometry invariant for roughly geodesic spaces. 

Admissible metrics naturally arise from geometric models for $\G$. Let $X$ be a geodesic hyperbolic space on which $\G$ acts isometrically, properly and cocompactly, and pick a basepoint $o\in X$. Then the orbit metric on $\G$, given by $|g,h|_o:=|go,ho|$, is admissible. (An innocuous issue is that $o$ might have non-trivial stabilizer. This is easily made irrelevant either by language, allowing pseudo-metrics instead of metrics, or by coarse bookkeeping.)

If $\G$ admits a CAT$(-1)$ geometric model, then the induced orbit metrics on $\G$ are strongly hyperbolic. The following theorem is a general statement to that effect, circumventing the delicate question whether a CAT$(-1)$ geometric model is always available.

\begin{thm}
There exist admissible metrics on $\G$ which are strongly hyperbolic.
\end{thm}

Implicitly, this was first proved in \cite{Min} by an involved construction of combinatorial flavour. In \cite{NS} we show that there are, in fact, \emph{natural} admissible metrics that are strongly hyperbolic. Namely, the Green metric defined by any symmetric and finitely supported random walk on $\G$ is, up to a rescaling, strongly hyperbolic \cite[Thm.6.1]{NS}.

\section{The Busemann flow for boundary actions of hyperbolic groups} 
\subsection{Preliminaries}\label{prelim} Let us start with some general facts on cocycles, flows, and KMS states for crossed-products. These matters are well-known, and they go back to Renault's foundational work \cite{Ren}. A minor difference is that we choose to work with reduced crossed-products, rather than full crossed-products.

Let $G$ be a discrete countable group acting by homeomorphisms on a compact Hausdorff space $\Omega$. The algebraic crossed-product $C(\Omega)\xalg G$ consists of finite sums of the form $\sum \phi_g g$, where $\phi_g\in C(\Omega)$ and $g\in G$. This is an algebra for the multiplication whose defining rule is that $(\phi g)(\psi h)=\phi(g.\psi)gh$. The \emph{reduced crossed-product} $C(\Omega)\xr G$ is the reduced C$^*$-completion of $C(\Omega)\xalg G$.

A \emph{flow} on a C$^*$-algebra $A$ is a strongly continuous group homomorphism $\sigma:\R\to \mathrm{Aut}(A)$. On crossed-products, cocycles give rise to flows, as follows. Consider a cocycle $c:G\to C(\Omega,\R)$, the real-valued continuous maps on $\Omega$. (Throughout this paper, the cocycle property is in the additive sense: $c(gh)=c(g)+g.c(h)$ for all $g,h\in G$.) Then there is a flow $\sigma^{c}$ on $C(\Omega)\xr G$, defined by the formula
\begin{align*}
\sigma^{c}_t\big(\sum \phi_g g\big)=\sum e^{itc(g)} \phi_g g
\end{align*}
on $C(\Omega)\xalg G$. 

Let $\sigma$ be a flow on a C*-algebra $A$, and $\beta\in \R$. A state $\omega$ on $A$ is said to be a $\beta$-\emph{KMS state} for $\sigma$ if 
\begin{align*}
\omega(b \sigma_{i\beta}(a))=\omega(ab)
\end{align*}
for all $a,b$ in a dense subalgebra of $\sigma$-entire elements of $A$. We refrain from defining the notion of $\sigma$-entire elements of $A$, except to mention that the $\sigma$-entire elements form a dense $*$-subalgebra of $A$. The parameter $\beta$ is called \emph{inverse temperature}.

Now consider the flow $\sigma^{c}$ on $C(\Omega)\xr G$, induced by a cocycle $c$ as above. Then all elements in $C(\Omega)\xalg G$ are entire. Let $\omega$ be a $\beta$-KMS state for $\sigma^{c}$. As for any state on $C(\Omega)\xr G$, the restriction of $\omega$ to $C(\Omega)$ defines a probability $\mu$ on $\Omega$. (Here, and in what follows, we use the term `probability' as a shorthand for `regular Borel probability measure'.) The KMS condition means that $\mu$ is $e^{\beta c}$-conformal, in the sense that 
\begin{align*}
\frac{d(g_*\mu)}{d\mu}=e^{\beta c(g)}
\end{align*} for each $g\in G$. Conversely, let $\mu$ be an $e^{\beta c}$-conformal probability on $\Omega$. Consider the state $\omega_\mu$ on $C(\Omega)\xr G$, defined by
\begin{align*}
\omega_\mu\big(\sum \phi_g g\big)=\int \phi_1 \:d\mu
\end{align*}
on $C(\Omega)\xalg G$. In other words, $\omega_\mu$ is the composition of the standard expectation $C(\Omega)\xr G\to C(\Omega)$ with $\mu$, viewed as a state on $C(\Omega)$. Then $\omega_\mu$ is a $\beta$-KMS state for the cocycle flow $\sigma^{c}$. 

The following result says that the previous construction is the only source of KMS states for $\sigma^{c}$, whenever $c$ satisfies a certain non-vanishing condition.

\begin{thm}[Kumjian - Renault \cite{KR}]\label{thm: KR}
Consider a cocycle flow $\sigma^{c}$ on $C(\Omega)\xr G$. Assume that, for all non-trivial $g\in G$, $c(g)$ is non-zero at each fixed point of $g$. Then every $\beta$-KMS state for $\sigma^{c}$ is of the form $\omega_\mu$ for some $e^{\beta c}$-conformal probability $\mu$ on $\Omega$.
\end{thm}

The non-vanishing condition could be thought of as a strong cohomological non-triviality. For if the cocycle $c:G\to C(\Omega,\R)$ is of the form $c(g)=g.\theta-\theta$, then $c(g)$ vanishes at each fixed point of $g$, for all $g\in G$.

\subsection{The boundary crossed product of a hyperbolic group}
Now let $\G$ be a hyperbolic group and consider the reduced crossed-product $C(\bd\G)\xr \G$, defined by the action of $\G$ on its boundary $\bd \G$. Endow $\G$ with a strongly hyperbolic, admissible metric.

A remarkable cocycle on $\G$ is the Busemann cocycle. To begin, there is the group Busemann cocycle, given by
\begin{align*}
b(g)(x)=2\la g,x \ra-|g| \qquad (x\in \G)
\end{align*}
for each $g\in \G$. Here, and in all that follows, the Gromov product is based at the identity, and we write $|g|$ for $|1,g|$, the distance from $g$ to the identity. The cocycle property for $b$ is easily checked. In fact, writing $b(g)(x)=|x|-|g^{-1}x|$ exhibits $b$ as a coboundary. 

Secondly, and more importantly for the purposes of this section, there is a boundary Busemann cocycle. By Theorem~\ref{thm: good}, the group Busemann cocycle extends, by continuity and as a continuous function, to the boundary. The boundary Busemann cocycle is given, for each $g\in \G$, by
\begin{align*}
b(g)(\xi)=2\la g,\xi \ra-|g| \qquad (\xi\in \bd\G).
\end{align*}
The boundary Busemann cocycle $b$ takes values in $C(\bd\G, \R)$, so it defines a flow $\sigma^b$ on $C(\bd\G)\xr \G$. 

On the other hand, by Theorem~\ref{thm: good}, once again, the Gromov product based at the identity induces a compatible metric
\begin{align*}
d(\xi_1,\xi_2)=e^{-\la \xi_1,\xi_2 \ra}
\end{align*} on $\bd \G$. Let $\mu$ be the probability on $\bd\G$ defined by normalizing the Hausdorff measure, and let $D$ denote the Hausdorff dimension of $\bd\G$.

\begin{thm}\label{thm a}
Consider the Busemann cocycle flow $\sigma^b$ on $C(\bd\G)\xr \G$. Then the probability $\mu$ induces a KMS state $\omega_\mu$ for $\sigma^b$, at inverse temperature $D$. If $\G$ is torsion-free, then $\omega_\mu$ is the unique KMS state for $\sigma^b$.
\end{thm}

\begin{proof} In order for $\omega_\mu$ to be a KMS state for $\sigma^b$ at inverse temperature $D$, we need to know that the probability $\mu$ is $e^{Db}$-conformal. Fix $g\in \G$. We have
\begin{align*}
-2\la gx,gy\ra= b(g^{-1})(x)+b(g^{-1})(y)-2\la x,y\ra
\end{align*}
for all $x,y\in \G$. This identity extends by continuity to the boundary, leading to
\begin{align*}
d(g\xi,g\eta)^2= e^{b(g^{-1})(\xi)}\:e^{b(g^{-1})(\eta)} \: d(\xi,\eta)^2
\end{align*}
for all $\xi,\eta\in \bd\G$. It follows, see \cite[Lem.8]{Nic}, that
\begin{align*}
\frac{d(g^*\mu)}{d\mu}=e^{Db(g^{-1})}
\end{align*} for each $g\in G$. Up to replacing $g$ by $g^{-1}$, this is means that $\mu$ is $e^{Db}$-conformal, as desired.

Now let us turn to the uniqueness statement, in which $\G$ is assumed to be torsion-free. We wish to apply the Kumjian - Renault criterion, so let us check that $b$ satisfies the non-vanishing condition of Theorem~\ref{thm: KR}. Let $g$ be a non-trivial element of $\G$. Then the following properties hold. Firstly, the infinite cyclic subgroup generated by $g$ is quasi-isometrically embedded in $\G$. Secondly, there are two distinct points $g^+, g^-\in \bd \G$ such that $g^n\to g^+$ and $g^{-n}\to g^-$ as $n\to \infty$. Thirdly, the points fixed by $g$ on the boundary are precisely $g^+$ and $g^-$.

For the group Busemann cocycle, we have 
\begin{align*}
b(g)(g^n)=|g^n|-|g^{n-1}|,\qquad  b(g)(g^{-n})=|g^{-n}|-|g^{-(n+1)}|=-b(g)(g^{n+1}).
\end{align*}

Letting $n\to \infty$, the second relation yields
\begin{align*}
b(g)(g^-)=-b(g)(g^+),
\end{align*}
while the first leads to
\begin{align*}
b(g)(g^+)=\lim_{n\to \infty} \big(|g^n|-|g^{n-1}|\big)=\lim_{n\to \infty} \frac{|g^n|}{n}
\end{align*}
by the discrete l'Hospital rule. But the right-hand limit is positive, as $g$ is undistorted, and we conclude that $b(g)(g^+)>0$ and $b(g)(g^-)<0$.

We deduce that a KMS state for $\sigma^b$ at inverse temperature $D'$ must be induced by a probability $\mu'$ on $\bd\G$ which is $e^{D'b}$-conformal. Results of Coornaert \cite{Coo}, and their generalizations to the roughly geodesic context by Blach\`ere, Ha\"{\i}ssinsky, and Mathieu \cite{BHM}, imply that $D'=D$ and $\mu'=\mu$.
\end{proof}

\section{The Haagerup cocycle for hyperbolic groups}\label{sec: haag}

\subsection{The Haagerup cocycle for free groups} Let $\F$ be a non-abelian free group. Then $\F$ admits a proper isometric action on a Hilbert space. This is due to Haagerup \cite{Haa}, up to a slight reinterpretation, and his elegant construction runs as follows. 

Consider the standard Cayley graph of $\F$ with respect to the free generators and their inverses. This is a regular undirected tree. Let $\vec{E}$ be the set of its \emph{oriented} edges. Then $\F$ acts on $\vec{E}$ in a natural way, and we may consider the corresponding orthogonal representation of $\F$ on $\ell^2(\vec{E})$. Next, we perturb this linear isometric action by a cocycle $c:\F\to \ell^2(\vec{E})$. Given $g\in \F$, let $c_g$ be the following function on $\vec{E}$: $c_g$ is supported on the geodesic path joining $g$ to the identity $1$, and for an oriented edge $e$ lying on this path we value $c_g(e)$ to be $+1$ or $-1$ according to whether $e$ points towards or away from $g$. In short:
\begin{align*}
c_g=\sum_{e\in \{1\to g\}} \delta_e-\sum_{e\in \{g\to 1\}} \delta_e
\end{align*}
The cocycle property, $c_{gh}=c_g+g.c_h$ for all $g,h\in \F$, can be seen by drawing the geodesic tripod defined by $1$, $g$, and $gh$, and noting that the oriented edges lying on the leg towards $g$ cancel out. Clearly, $c_g\in \ell^2(\vec{E})$ and
\begin{align*}
\|c_g\|^2_2=2|g|.
\end{align*}
In particular, the cocycle $c$ is proper: $\|c_g\|_2\to \infty$ as $g\to \infty$ in $\F$. It follows that the affine isometric action of $\F$ on $\ell^2(\vec{E})$ given by $(g,\phi)\mapsto g.\phi+c_g$ is proper. Note that this construction applies, in fact, to any space $\ell^p(\vec{E})$ for $p\in [1,\infty)$.

We wish to adapt Haagerup's construction to a general hyperbolic context, and we start by recasting the above cocycle in a more convenient form. Firstly, we think of the oriented edge-set $\vec{E}$ as the set $\{(x,y)\in \F\times\F: |x,y|=1\}$. Secondly, we note that the cocycle $c$ can be described by in metric terms by the following formula:
\begin{align}\label{eq: newform}
c_g(x,y)=\la g,x\ra-\la g,y\ra \tag{$\dagger$}
\end{align}
Recall that $\la\cdot,\cdot\ra$ denotes the Gromov product based at the identity. In this form, the cocycle property is even more transparent: writing
\begin{align*}
c_g(x,y)=\tfrac{1}{2}\big(|x|-|g^{-1}x|\big)-\tfrac{1}{2}\big(|y|-|g^{-1}y|\big)
\end{align*}
we obtain the coboundary formula $c_g=F-g.F$, for $F(x,y)=\tfrac{1}{2}\big(|x|-|y|\big)$.

\subsection{The Haagerup cocycle for hyperbolic groups}
Let $\G$ be a hyperbolic group, which we may assume to be non-elementary. Endow $\G$ with a strongly hyperbolic admissible metric. We also consider a coarse relative of the underlying set we have used in the free group case. Namely, let
\begin{align*}
\Delta=\big\{(x,y)\in \G\times\G: K-C\leq |x,y|\leq K+C\big\}
\end{align*} 
where $C\geq 0$ is a rough geodesic constant, and $K>0$ is another constant. For the purposes of the following theorem, we ask that $K>2C$. Note that $\Delta$ is non-empty. This can be seen by choosing a convenient point along a rough geodesic from the identity to some sufficiently remote group element. 

The group $\G$ acts on $\Delta$, by $g.(x,y)=(gx,gy)$. Let $c_g$ be defined on $\Delta$ by the metric formula \eqref{eq: newform}. Then $c$ is a cocycle for $\G$, for the same reasons as explained above. 

\begin{thm}\label{thm b}
For large enough $p\in [1,\infty)$, the affine isometric action of $\G$ on $\ell^p(\Delta)$ given by $(g,\phi)\mapsto g.\phi+c_g$ is well-defined and proper.
\end{thm}

\begin{proof} 
For the action to be well-defined, we need to have $c_g\in \ell^p(\Delta)$ for each $g\in \G$. An application of the mean value theorem to the function $t\mapsto e^{-t}$ yields
\begin{align*}
\big|e^{-\la g,x\ra}-e^{-\la g,y\ra}\big|\geq e^{-\max\{\la g,x\ra, \la g,y\ra\}}\: |\la g,x\ra-\la g,y\ra|.
\end{align*} 
The left-hand side is at most $e^{-\la x,y\ra}$, thanks to strong hyperbolicity. On the right-hand side, both $\la g,x\ra$ and $\la g,y\ra$ are at most $|g|$. It follows that
\begin{align*}
|c_g(x,y)|\leq e^{|g|}\:e^{-\la x,y\ra}.
\end{align*}
We complete the argument by showing that $e^{-\la \cdot,\cdot\ra}\in \ell^p(\Delta)$ for large enough $p\in [1,\infty)$. If $(x,y)\in \Delta$ then $\la x,y\ra\geq |x|-|x,y|\geq |x|-(K+C)$. We deduce that
\begin{align*}
\sum_{(x,y)\in \Delta}e^{-p\:\la x,y\ra}\leq C_1 \sum_{(x,y)\in \Delta} e^{-p|x|}\leq C_2 \sum_{x\in \G}e^{-p|x|}\end{align*}
and the latter sum converges when $p$ is large enough. 

For the action to be proper, we need to argue that $\|c_g\|_p\to \infty$ as $g\to \infty$ in $\G$. In fact, we show that there are constants $C',C''>0$, depending only on $K$, $C$, and $p$, such that
\begin{align*}
\|c_g\|^p_p\geq C'|g|-C''
\end{align*}
for each $g\in \G$.

Let $\gamma:[a,b]\to\G$ be a rough geodesic joining the identity to $g$. The basic idea is that $|c_g(x,y)|$ is roughly $|x,y|$ whenever $x$ and $y$ lie on $\gamma$, and that we can find about $|g|/K$ pairs of points on $\gamma$ that belong to $\Delta$. Now let us be precise.

Consider the elements $\gamma(t_i)\in \G$ arising from a partition $a=t_0<\ldots<t_n\leq b$ into $n$ intervals of length $K$, and a remainder of length less than $K$. Then $|\gamma(t_i),\gamma(t_{i+1})|$ is within $C$ of $|t_i-t_{i+1}|=K$, so $(\gamma(t_i),\gamma(t_{i+1}))\in \Delta$. Also, $c_g(\gamma(t_i),\gamma(t_{i+1}))$ can be written as
\begin{align*}
\tfrac{1}{2}\big(|\gamma(a),\gamma(t_i)|-|\gamma(b),\gamma(t_i)|\big)-\tfrac{1}{2}\big(|\gamma(a),\gamma(t_{i+1})|-|\gamma(b),\gamma(t_{i+1})|\big)
\end{align*}
which is within $2C$ of
\begin{align*}
\tfrac{1}{2}\big(|a-t_i|-|b-t_i|\big)-\tfrac{1}{2}\big(|a-t_{i+1}|-|b-t_{i+1}|\big)=t_{i+1}-t_i=K.
\end{align*}
In particular, $c_g(\gamma(t_i),\gamma(t_{i+1}))\geq K-2C>0$, according to our assumption on $K$. Hence 
\begin{align*}
\|c_g\|^p_p= \sum_{(x,y)\in \Delta} \big|c_g(x,y)\big|^p\geq \sum_{i=0}^{n-1} \big|c_g(\gamma(t_i),\gamma(t_{i+1}))\big|^p\geq (K-2C)^pn.
\end{align*} 

On the other hand, we can relate $n$ and $|g|$. The way we defined the partition implies that $K(n+1)>b-a$, and $b-a\geq |g|-C$ by using the rough geodesic property at the endpoints. Therefore $n\geq \big(|g|-(K+C)\big)/K$, and the desired claim follows.
\end{proof}

We end by pointing out that the cocycle used in \cite{Nic} is the boundary analogue of \eqref{eq: newform}, namely $c_g(\xi,\eta)=\la g,\xi\ra-\la g,\eta\ra$ for $\xi,\eta\in \bd \G$.

\begin{ack}
I thank Jean Renault for discussions around \S\ref{prelim}.
\end{ack}

%%%%%%%%%%%%%%%%%%%%%%%%%%%%%%%%%%%%%%%%

\end{document}